\documentclass[12pt]{extarticle}

\usepackage[margin=1.5in]{geometry}  
\usepackage{graphicx}              
\usepackage{amsmath}               
\usepackage{amsfonts}              
\usepackage{amsthm}                
\usepackage{framed}

\newtheorem{thm}{Theorem}[section]
\newtheorem{lem}[thm]{Lemma}

\begin{document}

\nocite{*}

\title{\bf An Explicit Result for \\ Primes Between Cubes}

\author{\textsc{Adrian W. Dudek} \\ 
Mathematical Sciences Institute \\
The Australian National University \\ 
\texttt{adrian.dudek@anu.edu.au}}
\date{}

\maketitle

\begin{abstract}
We prove that there is a prime between $n^3$ and $(n+1)^3$ for all $n \geq \exp(\exp(33.217))$. Our new tool which we derive is a version of Landau's explicit formula for the Riemann zeta-function with explicit bounds on the error term. We use this along with other recent explicit estimates regarding the zeroes of the Riemann zeta-function to obtain the result. Furthermore, we show that there is a prime between any two consecutive $m$th powers for $m \geq 4.971 \times 10^9$.
\end{abstract}

\section{Introduction}

Legendre's conjecture is the assertion that there is at least one prime between any two consecutive squares. Confirmation of this seems to be out of reach, for applying modern techniques on the assumption of even the Riemann hypothesis does not suffice in forming a proof (see Titchmarsh's \cite{titchmarsh1986theory} classic text for a discussion). It is thus the aim of this paper to study the weaker problem of primes between cubes, where some progress has already been made.

Consider first the more general problem of showing the existence of at least one prime in the interval $(x,x+x^{\theta})$ for some $\theta \in (0,1)$ and for all sufficiently large $x$.  These are \textit{short} intervals, so called for the relative (compared to the size of $x$) length of such an interval tends to zero as $x$ goes to infinity. 

In 1930, Hoheisel \cite{hoheisel} was able to solve the problem for $\theta = 1-1/33000$, that is, there will be a prime in the interval

$$(x,x+x^{32999/33000})$$
for all sufficiently large $x$. His idea relied on having asymptotic estimates on the distribution of zeroes of the Riemann zeta-function $\zeta(s)$, namely a zero-free region and a zero-density estimate. Landau's explicit formula for the Riemann zeta-function then allows a connection to be made between the zeroes and the primes. Using Hoheisel's ideas, Ingham \cite{ingham} was able to prove a more general theorem, specifically that if one has a bound of the form

$$\zeta(1/2+it) = O(t^c )$$
for some $c>0$, then one can take 

$$\theta = \frac{1+4c}{2+4c}+\epsilon.$$

This result arises through using the bound for $\zeta(1/2+it)$ to construct a zero-density estimate, and this in turn furnishes a value for $\theta$ through the explicit formula. Notably, Hardy and Littlewood were able to give a value of $c=1/6+\epsilon$, which corresponds to $\theta = 5/8+\epsilon$. From this, one sets $x = n^3$ and it follows that for all sufficiently large $n$ there exists a prime in the interval

$$(x,x+x^{5/8}) = (n^3,n^3+n^{15/8+ \epsilon}) \subset (n^3, (n+1)^3).$$
That is, there is a prime between any two consecutive cubes so long as these are sufficiently large. The reader should, however, note that consideration of the interval

$$(x,x+3x^{2/3})$$
is sufficient for primes between cubes and as such is the interval we use throughout this paper.

We wish to make the above result explicit, in that we might determine numerically a lower bound for which this result would hold onwards. By working through that of Ingham \cite{ingham}, we can do this thanks to Ford's \cite{ford} explicit zero-free region, Ramar\'{e}'s \cite{ramare} explicit zero-density theorem and an effective version of the explicit formula (see Theorem \ref{explicitformula}). We employ these estimates to prove our main theorem.

\begin{thm} \label{maintheorem}
There is a prime between $n^3$ and $(n+1)^3$ for all $n \geq \exp(\exp(33.217))$.
\end{thm}

We should note that a result has been given by Cheng \cite{cheng}, in which he purports to show the above theorem for the range $n \geq \exp(\exp(15))$. We should, however, note that he incorrectly goes from

$$n^3 \geq \exp(\exp(45))$$
to 

$$n \geq \exp(\exp(15))$$
in establishing his result. There are some other errors also, notably in his proof of Theorem 3 in his paper \cite{cheng}, the first inequality sign is backwards and he has used Chebyshev's $\psi$-function instead of the $\theta$-function. 

On the other hand, we have the result of Ramar\'{e} and Saouter \cite{ramaresaouter} that every interval 

$$(x(1-\Delta^{-1}),x]$$
contains a prime for $x > x'$ and for various values of $\Delta(x')$ (given in Table 1 of their paper). It is a straightforward task to use this table to verify that there is a prime in $(x,x+3x^{2/3})$ for all $x \leq e^{60}$. One simply works through the table whilst checking that the Ramar\'{e}--Saouter interval is contained in the short interval. 

We should also mention the striking result of Baker, Harman and Pintz \cite{bakerharmanpintz}, that the interval $(x,x+x^{0.525})$ contains a prime for all sufficiently large $x$. This is tantalisingly close to $\theta = 1/2$, which would furnish a proof of Legendre's conjecture with at most finitely many exceptions.

\section{An effective version of the explicit formula}

\subsection{The result}

Our considerations begin with the von Mangoldt function:

\begin{displaymath}
   \Lambda(n) = \left\{
     \begin{array}{ll}
       \log p  & : \hspace{0.1in} n=p^m, \text{ $p$ is prime, $m \in \mathbb{N}$}\\
       0   & : \hspace{0.1in} \text{otherwise.}
     \end{array}
   \right.
\end{displaymath} 
As usual, we introduce the sum $\psi(x) = \sum_{n \leq x} \Lambda(n)$ to study the distribution of the prime numbers. The usefulness of this function presents itself in the \textit{explicit formula}

\begin{equation}
\psi(x) = x - \sum_{\rho} \frac{x^\rho}{\rho}-\log 2\pi - \frac{1}{2} \log(1-x^{-2})
\end{equation}
where $x$ is any positive non-integer and the sum is over all nontrivial zeroes $\rho$ of $\zeta(s)$ (see Titchmarsh \cite{titchmarsh1986theory} for details). One can see that feeding information regarding the zeroes of $\zeta(s)$ into the above formula yields estimates for the prime powers, allowing us then to obtain estimates regarding the primes. However, the explicit formula as seen above relies on estimates over \textit{all} of the nontrivial zeroes of $\zeta(s)$ and so is impractical for certain applications.

The Riemann Hypothesis asserts the fixedness of these zeroes to the line $\Re(s) = 1/2$. As this grand statement is yet to be confirmed, we often find more use in a \textit{truncated} version of the explicit formula in which one inputs information regarding the zeroes $\rho$ up to some height $T$, that is, with $| \Im \rho | < T$. It is our first intention to provide such a formula but with an explicit error term. Our result is as follows.

\begin{thm} \label{explicitformula}
Let $x>e^{60}$ be half an odd integer and suppose that $T \in (50,x)$ is not the ordinate of any zero of $\zeta(s)$. Then

\begin{equation} \label{explicitequation}
\psi(x) = x - \sum_{|\Im \rho| < T} \frac{x^{\rho}}{\rho} + O^*\Big( \frac{ 2 x \log^2 x}{T} \Big)
\end{equation}
where the notation $f = g + O^* (h)$ means $|f-g| < h$.
\end{thm}

In the remainder of this paper, unless otherwise mentioned, we shall assume that $x$ and $T$ are in the ranges specified in Theorem \ref{explicitformula}, and that $x$ is half an odd integer (this assumption minimises the error term in the above theorem, as will become evident in the proof of Lemma $\ref{boundbigsum}$).

The author attempted to locate an effective explicit formula in the literature and found that which was given by Skewes \cite{skewes1955difference}, though the error term here is of better order. Liu and Wang \cite{liuwang} give a version of Theorem \ref{explicitformula} with an improved constant, but holding only for $T$ as a certain function of $x$, which is useful for estimates on Goldbach's weak conjecture but not for our application.

The method of proof for the explicit formula is well-known: we employ Perron's formula to express $\psi(x)$ as a contour integral over a vertical line. We then truncate this integral to one over a finite line segment. This is where we will pick up the bulk of the error term, and so more precision is needed here than anywhere else in this paper. Our line of integration is then shifted so as to acknowledge the residues and introduce the sum which accompanies $\psi(x)$ in Theorem \ref{explicitformula}. The crux of our working involves keeping track of the errors.

We proceed as laid out in Davenport's \cite{davenport1980multiplicative} well known expository text, though by working carefully we will obtain an explicit result. We should also note that the frame of this derivation can be applied to other arithmetic functions attached to an appropriate Dirichlet series. 

\subsection{Truncating the Line Integral}

For $c>0$, we define the contour integral

$$\delta(x) = \frac{1}{2 \pi i} \int_{c-i\infty}^{c+i\infty} \frac{x^s}{s} ds.$$
A good exercise (see Murty's \cite{murtyproblems} problem book) for budding students of analytic number theory (or complex analysis) is to show that

\begin{equation*}
\delta(x) =
\left\{
	\begin{array}{ll}
		0  & \mbox{if } 0 < x < 1 \\  
		1/2 & \mbox{if } x=1 \\  
		1 & \mbox{if } x > 1.
	\end{array}
\right.
\end{equation*}
The importance of this integral becomes apparent when one wishes to study the sum of an arithmetic function up to some value $x$, particularly when that function is generated by a Dirichlet series. In our case, we consider the following for a positive non-integer $x$ and $c>1$:

\begin{eqnarray*}
\psi(x) = \sum_{n \leq x} \Lambda(n) & = &  \sum_{n = 1}^{\infty} \Lambda(n) \delta\Big(\frac{x}{n} \Big) \\
& = & \sum_{n=1}^{\infty} \Lambda(n) \bigg[ \frac{1}{2 \pi i} \int_{c-i \infty}^{c+i \infty} \Big( \frac{x}{n} \Big)^s \frac{ds}{s} \bigg] \\
& = & \frac{1}{2 \pi i} \int_{c-i \infty}^{c+i \infty} \Big( \sum_{n=1}^{\infty} \frac{\Lambda(n)}{n^s} \Big) \frac{x^s}{s} ds.
\end{eqnarray*}
Notice that keeping $c>1$ gives absolute convergence to the series in the above equation, and thus justifies the interchange of integration and summation. The Dirichlet series in the above equation is known to be equal to $-\zeta'(s)/\zeta(s)$, and so we have that

$$\sum_{n \leq x} \Lambda(n) = \frac{1}{2 \pi i}  \int_{c-i \infty}^{c+i \infty} \Big( - \frac{\zeta'(s)}{\zeta(s)} \Big) \frac{x^s}{s} ds.$$
In a more general form this is known as Perron's formula. We may thus estimate the sum of the von Mangoldt function through some knowledge of certain analytic properties of $\zeta'(s)/\zeta(s)$. Our first step is to truncate the path of the integral to a finite segment, namely $(c-iT, c+iT)$. We define for $T > 0$ the truncated integral

$$I(x,T) =  \frac{1}{2 \pi i} \int_{c-i T}^{c+i T} \frac{x^s}{s} ds.$$

The next lemma is a variant of the first lemma in Davenport \cite[Ch.17]{davenport1980multiplicative}, and will bound the induced error term upon estimating $\delta(x)$ by $I(x,T)$. The proof is omitted here, for it is almost identical except in that we maintain the multiplicative constants $1/2 \pi$ throughout which arise from the contour integrals. 

\begin{lem}
For $x>0$ with $x \neq 1$, $c>0, T>0$ we have

\begin{equation*}
\delta(x) = I(x,T) + O^*\Big( \frac{x^c}{\pi T | \log x |}  \Big).
\end{equation*}

\end{lem}

We can now employ this bound to estimate $\psi(x)$ in the following way.

\begin{eqnarray*}
\psi(x) & = &  \sum_{n = 1}^{\infty} \Lambda(n) \delta\Big(\frac{x}{n} \Big) \\ \\
& = &  \sum_{n = 1}^{\infty} \Lambda(n) \Big[I(\frac{x}{n},T) + O^*\Big(\frac{1}{\pi T} \Big(\frac{x}{n}\Big)^c  \Big| \log \frac{x}{n} \Big|^{-1} \Big)\Big] \\ \\
& = &  \frac{1}{2 \pi i} \int_{c-i T}^{c+i T} - \frac{\zeta'(s)}{\zeta(s)} \frac{x^s}{s} ds + \frac{1}{\pi T} O^* \Big( \sum_{n=1}^{\infty}  \Lambda(n) \Big(\frac{x}{n}\Big)^c \Big| \log \frac{x}{n} \Big|^{-1} \Big).
\end{eqnarray*}
We proceed to bound the sum in the above formula, by splitting it up and estimating it in parts. 

\begin{lem} \label{boundbigsum}
Let $x > e^{60}$ be half an odd integer and set $c = 1+1/\log x$. Then

\begin{equation}
\sum_{n=1}^{\infty} \Lambda(n) \Big( \frac{x}{n} \Big)^c \Big| \log \frac{x}{n} \Big|^{-1} < 2.8 x \log^2 x.
\end{equation}

\end{lem}

\begin{proof}
Some care needs to be taken here. When $x$ and $n$ are quite close, the reciprocal $\log$ will become large. Thus, we introduce the parameter $\alpha \in (1,2)$ and break up the infinite sum:

\begin{eqnarray*}
\sum_{n=1}^{\infty} & = & \sum_{n=1}^{[x/\alpha]} + \sum_{n=[x/\alpha] + 1}^{[x]-1} + \sum_{n = [x]}^{[x]+1} + \sum_{n = [x]+2}^{[\alpha x]}+\sum_{n=[\alpha x]+1}^{\infty} 
\end{eqnarray*}
On the right side of the above formula, denote the $i$th sum by $S_i$. The reader should be convinced by this division; $S_3$ deals with the most inflated terms, namely when $n$ is either side of $x$. Then $S_2$ and $S_4$ measure the remainder of the region which is \textit{close} to $x$. We also note that $S_1$ and $S_5$ contribute little and can be estimated almost trivially.

Considering the range of $n$ in $S_1$ and $S_5$, we have

$$\Big| \log \frac{x}{n} \Big| > \log \alpha.$$
Inserting this into these sums, pulling out terms which are independent on $n$, and extending the range of summation to $\mathbb{N}$ we arrive at

$$S_1 + S_5 < \frac{x^c}{\log \alpha} \sum_{n=1}^{\infty}\frac{\Lambda(n)}{n^c} = \frac{x^c}{\log \alpha} \Big( -\frac{\zeta'(c)}{\zeta(c)} \Big).$$
We then use Lemma 3.1 from \cite{ford} to obtain

\begin{equation} \label{s1s5}
S_1+S_5 < \frac{ex \log x}{\log \alpha}.
\end{equation} 

We now turn our attention to $S_3$, which is the sum of only two things. It follows, using the fact that $[x]=x-1/2$ and the trivial bound $\Lambda(n) \leq \log n$, that

\begin{eqnarray*}
S_3 & = & \Big( \frac{x}{x-\frac{1}{2}} \Big)^c \Lambda(x-1/2) \Big| \log \frac{x}{x-\frac{1}{2}} \Big| +  \Big( \frac{x}{x+\frac{1}{2}} \Big)^c \Lambda(x+1/2) \Big| \log \frac{x}{x+\frac{1}{2}} \Big| \\
& < & 4x \log x \Big( \frac{x}{x-\frac{1}{2}} \Big)^c.
\end{eqnarray*}
This will actually be of little consequence to the final sum (as we will soon see), and so we feel no remorse in collecting here the weaker bound

\begin{equation} \label{s3}
S_3 < 5 x \log x.
\end{equation}

In consideration of $S_2$, we estimate $x/n < \alpha$ and $\Lambda(n) \leq \log n$ to get

$$S_2 < \alpha^c \log x \sum_{n=[x/\alpha]+1}^{[x]-1} \Big| \log \frac{x}{n} \Big|^{-1}.$$
If we let $n = [x] - v$, then the problem becomes that of summing over $v=1,2,\ldots, [x]-[x/\alpha]-1$. We have

$$\Big| \log \frac{x}{n} \Big| = \log \frac{x}{n} > \log \frac{[x]}{n} = - \log \Big(1- \frac{v}{[x]} \Big) > \frac{v}{[x]}$$
and thus

$$S_2 < \alpha^c x \log x \sum_{v=1}^{[x]-[x/\alpha]-1} \frac{1}{v}.$$
One can estimate this by the known bound $\sum_{n \leq x} 1/n \leq \log x + \gamma + 1/x$ where $\gamma \approx 0.5772\ldots$ is Euler's constant to get:

\begin{equation} \label{s2}
S_2 < \alpha^c x \log x \Big( \log(x - x/\alpha)) + \gamma + \frac{1}{x-x/\alpha} \Big).
\end{equation}

The sum $S_4$ is dealt with similarly to get

\begin{equation} \label{s4}
S_4 < 2 \frac{\log( \alpha x)}{3 - \alpha} \Big( \log( \alpha x - x) + \gamma + \frac{1}{\alpha x - x} \Big).
\end{equation}

Finally, one may combine $(\ref{s1s5}), (\ref{s3}), (\ref{s2}), (\ref{s4})$ to get an inequality of the form

\begin{equation}
\sum_{n=1}^{\infty} \Lambda(n) \Big( \frac{x}{n} \Big)^c \Big| \log \frac{x}{n} \Big|^{-1} < f(\alpha, x).
\end{equation}
The result follows now from choosing $\alpha = 1.194$ and letting $x>e^{60}$.

\end{proof}

\subsection{Error Bounds}

The immediate result of Lemma $\ref{boundbigsum}$ is that

\begin{eqnarray*}
\psi(x) & = &   \frac{1}{2 \pi i} \int_{c-i T}^{c+i T} - \frac{\zeta'(s)}{\zeta(s)} \frac{x^s}{s} ds + O^* \Big( \frac{2.8 x \log^2 x}{\pi T} \Big)
\end{eqnarray*}
for $x>e^{60}$, $c=1+1/\log x$ and $T>0$. We now look to shifting the line of integration so that we might involve the residues of the integrand. In doing so, we incur errors which only slightly increase the above error term. That is, the bulk of the error has already been obtained, and so we can be excused for not pursuing the best possible bound in the remainder of this section. Let $U>0$ and define the line segments 

$$C_1 = [c-iT,c+iT] \hspace{1in} C_2 = [c+iT,-U+iT]$$ 

$$C_{3} = [-U+iT,-U-iT] \hspace{0.8in} C_{4} = [-U-iT,c-iT]$$
and their union $C$ along with the corresponding integrals

$$I_i = \frac{1}{2 \pi i} \int_{C_i} - \frac{\zeta'(s)}{\zeta(s)} \frac{x^s}{s} ds.$$
We also denote by $I$ the integral around the rectangle $C$. Note that we need to account for the fact that while $T$ is stipulated not to be the ordinate of a zero of $\zeta(s)$, it might be undesirably close to such. We show in Lemma $\ref{choicet}$ that there is always some good choice of $T$ nearby, and so some work will be required later to shift our horizontal paths. Also note that any work we do in bounding $I_2$ will also hold for $I_{4}$ and so it follows that

\begin{equation} \label{psi}
| \psi(x) - I | < 2 |I_2|+|I_{3}| + 2.8 \frac{x \log^2 x}{\pi T}.
\end{equation}

One can use Cauchy's theorem (see Davenport \cite{davenport1980multiplicative} for full details) to show that

\begin{equation*} 
I = x - \sum_{|\gamma| < T} \frac{x^{\rho}}{\rho} - \frac{\zeta'(0)}{\zeta(0)}+\sum_{0 < 2m < U} \frac{x^{-2m}}{2m}.
\end{equation*}
Noting that the rightmost summation is a partial sum of the series for $\log(1-x^{-2})/2$, we can write that

$$\psi(x) = x - \sum_{\rho} \frac{x^{\rho}}{\rho} + O^*\big(E(x,T,U)\big)$$
where 

\begin{equation} \label{bigerror}
E(x,T,U) < \frac{\zeta'(0)}{\zeta(0)} + \frac{1}{2} \log(1-x^{-2}) + 2 |I_2|+|I_{3}| + 2.8 \frac{x \log^2 x}{\pi T}.
\end{equation}
It so remains to bound $|I_2|$ and $|I_{3}|$ by deriving and making use of explicit estimates for $| \zeta'(s)/\zeta(s)|$ in appropriate regions.

\subsection{Explicit Bounds for $|\zeta'(s)/\zeta(s)|$}

We first establish a bound on the lengths of the rectangle $C$ that intersect with the half-plane $\sigma \leq -1$. Our contour, or rather $U$, is chosen so that we might avoid the poles of $\tan \pi s/2$ which occur at the odd integers. 

\begin{lem}  \label{bound1}
Let $U$ be an even integer. Then

\begin{equation*}
\Big| \frac{\zeta'(s)}{\zeta(s)} \Big| < 9 + \log | s|
\end{equation*}
on the intersection of $C$ with $\sigma \leq -1$.

\end{lem}

\begin{proof}

Consider the logarithmic derivative of the functional equation

$$\frac{\zeta'(1-s)}{\zeta(1-s)} = -\log 2\pi - \frac{1}{2} \pi \tan \frac{\pi s}{2} + \frac{\Gamma'(s)}{\Gamma(s)} + \frac{\zeta'(s)}{\zeta(s)}.$$
Let $\sigma \geq 2$ (so that $1-\sigma \leq -1$) and notice that $| \frac{1}{2} \pi \tan \frac{\pi s}{2}| < 2$ so long as $s$ is distanced by at least $1$ from odd integers on the real axis (this justifies our choice of $U$). We can then use

\begin{equation}  \label{gamma}
\frac{\Gamma'(s)}{\Gamma(s)} = \log s - \frac{1}{2s} - \int_0^{\infty} \frac{[u]-u+1/2}{(u+s)^2}du
\end{equation}
to bound $|\Gamma'(s)/\Gamma(s)|$ trivially. The result then follows by observing that

\begin{equation} \label{trivialzetabound}
\Big| \frac{\zeta'(s)}{\zeta(s)} \Big| \leq - \frac{\zeta'(2)}{\zeta(2)}  < \frac{3}{5}
\end{equation}
and putting it all together.

\end{proof}

We now look to the harder task of establishing a bound over the region that includes the critical strip, as is essential for the estimation of $I_2$. 

\begin{lem} \label{zetaaszeroes}
Let $s = \sigma +it$ where $\sigma > -1$ and $t >50$. Then

\begin{equation}
\frac{\zeta'(s)}{\zeta(s)} = \sum_{\rho} \Big(\frac{1}{s-\rho}-\frac{1}{2+it-\rho} \Big) + O^*(2 \log t),
\end{equation}

\end{lem}

\begin{proof}
We start with the equation (see 12.8 of Davenport \cite{davenport1980multiplicative})

\begin{equation} \label{zetazeroes}
-\frac{\zeta'(s)}{\zeta(s)} = \frac{1}{s-1} - B -\frac{1}{2} \log \pi + \frac{\Gamma'(\frac{s}{2}+1)}{2 \Gamma(\frac{s}{2}+1)} - \sum_{\rho} \Big(\frac{1}{s-\rho}+\frac{1}{\rho} \Big)
\end{equation}
where $B = -\gamma/2 - 1 +\frac{1}{2} \log 4 \pi$. Successively, we set $s_0 = 2+it$ and $s = \sigma +it$ and then find the difference between the two expressions. The terms involving the $\Gamma$-function are dealt with using (\ref{gamma}), whereas the rest are estimated either trivially or with (\ref{trivialzetabound}) to arrive at the result.

\end{proof}

We can estimate the sum in Lemma \ref{zetaaszeroes} by breaking it into two smaller sums $S_1$ and $S_2$, where $S_1$ ranges over the zeroes $\rho = \beta+i \gamma$ with $|\gamma-t| \geq 1$ and $S_2$ is over the remaining zeroes. 

\begin{lem}
Let $s = \sigma +it$, where $\sigma > -1$ and $t >50$. Then

$$S_1 =  \sum_{|t-\gamma| \geq 1} \Big(\frac{1}{s-\rho}-\frac{1}{2+it-\rho} \Big)  = O^*(16 \log t).$$
\end{lem}

\begin{proof}

We can estimate the summand as follows:

\begin{equation}
\Big| \frac{1}{s-\rho} - \frac{1}{2+it-\rho} \Big| < \frac{3}{(t - \gamma )^2}.
\end{equation} 
We then have that

$$S_1 < \sum_{|t - \gamma | \geq 1} \frac{3}{(t-\gamma)^2} \leq \sum_{\rho} \frac{6}{1+(t-\gamma)^2}.$$
By letting $\sigma' = 2$, taking real parts in (\ref{zetazeroes}) and estimating as before one has

$$\sum_{\rho} \Re \Big( \frac{1}{s-\rho} + \frac{1}{\rho} \Big) <  \frac{2}{3} \log t$$
for $t > 50$. We then use the two simple facts

$$\Re \Big( \frac{1}{s-\rho} \Big) = \frac{2 - \beta}{(2-\beta)^2+(t-\gamma)^2} > \frac{1}{4+(t-\gamma)^2}$$
and

$$\Re \Big( \frac{1}{\rho} \Big) = \frac{\beta}{|\rho|^2} > 0$$
to get

$$\sum_{\rho} \frac{1}{4+(t-\gamma)^2} < \frac{2}{3} \log t.$$
Putting it all together we have

\begin{eqnarray*}
S_1 & < & \sum_{\rho} \frac{6}{1+(t-\gamma)^2}\\
&<& 24 \sum_{\rho} \frac{1}{4+(t-\gamma)^2} \\
& < & 16 \log t.
\end{eqnarray*}

\end{proof}

We now wish to estimate the remaining sum

$$S_2 = \sum_{|\gamma-t|<1} \Big( \frac{1}{s-\rho} - \frac{1}{2+it - \rho} \Big).$$
To do this, we first note that as $|2+it-\rho| > 1$ the contribution of the second term to the sum can be estimated trivially by 
 
$$N(t+1)-N(t-1)$$
where $N(T)$ denotes the number of zeroes of $\zeta(s)$ in the critical strip up to height $T$. We can use Corollary 1 of Trudgian \cite{trudgianargument} with $T_0 = 50$ to get

\begin{equation} \label{boundzerocount}
N(t+1) - N(t-1) < \log t
\end{equation}
so long as $t > 250000$. Using \textsc{Mathematica}, we can additionally verify this for $50<t \leq 250000$. At this point we have that

$$\frac{\zeta'(s)}{\zeta(s)} = \sum_{| t-\gamma|<1} \frac{1}{s-\rho} + O^*(19 \log t).$$

 Finally, we are concerned with bounding the magnitude of

$$S_2' = \sum_{|\gamma-t|<1} \frac{1}{s-\rho}. $$
Of course, the problem here is that $s$ might be close to a zero $\rho$, and this will give us trouble when we attempt to bound the line integrals. We search instead for a better value of $t$, say $t_0 \in (t-1,t+1)$, which will give a better bound. We will use this in the next section to shift our horizontal line of integration to a better height.

\begin{lem} \label{choicet}
Let $t> 50$. There exists $t_0 \in (t-1,t+1)$ such that

\begin{equation} \label{bound2}
\bigg| \sum_{|\gamma-t|<1} \frac{1}{(\sigma+it_0)-\rho} \bigg| < \log^2 t + \log t
\end{equation}

\end{lem}

\begin{proof}
By $(\ref{boundzerocount})$, there are at most $ \log t$ terms in $S_2'$. As such, we can partition the region into no more than $\log t +1$ zero-free regions.  We can see that there will always be such a region of height 

$$\frac{2}{ \log t + 1}$$
and choosing the midpoint, say $t_0$, of this region will guarantee a distance of 

$$\frac{1}{ \log t +1}$$
from any zero. As such, we have, letting $s = \sigma + i t_0$, that

$$\sum_{|\gamma-t|<1} \frac{1}{|s- \rho |} \leq \sum_{|\gamma-t|<1} \frac{1}{|\gamma - t|} \leq \sum_{|\gamma-t|<1} ( \log t +1) \leq   \log^2 t +  \log t.$$

\end{proof}

Finally, we can put the previous three lemmas together to get the following.

\begin{lem}
Let $\sigma > - 1, t > 50$. There exists $t_0 \in (t-1,t+1)$ such that

$$\Big| \frac{\zeta'(\sigma+it_0)}{\zeta(\sigma+it_0)} \Big| < \log^2 t+ 20\log t.$$
\end{lem}

That is, if our contour is somewhat close to a zero, we can shift it slightly to a region where we have good bounds.

\subsection{Integral Estimates}

We now bound the error term $E(x,T,U)$ in $(\ref{bigerror})$, by estimating each integral trivially. Using Lemma \ref{bound1}, we have

\begin{eqnarray*}
|I_3| & = &\frac{1}{2\pi} \Big| \int_{-U-iT}^{-U+iT} - \frac{\zeta'(s)}{\zeta(s)} \frac{x^s}{s} ds \Big|\\ \\ 
& < &  \int_{-T}^{T} \frac{9 + \log \sqrt{U^2+T^2}}{2 \pi x^U T}dt \\ \\
&=&   \frac{18+2\log \sqrt{U^2+T^2}}{2 \pi x^U}.
\end{eqnarray*}
We save this, for soon we will combine our estimates and bound them in unison upon an appropriate choice for $U$. Consider now the problem of estimating $I_2$, and the issue that $T$ might be close to the ordinate of a zero. From Lemma \ref{choicet}, there exists some $T_0 \in (T-1,T+1)$ that we should integrate over instead. We thus aim to shift the line of integration from $C_2$ to 

$$C_2 ' = [-U+iT_0, c +iT_0].$$

It follows from Cauchy's theorem that

$$|I_2| < \sum_{T-1 < \rho < T+1} \Big| \frac{x^{\rho}}{\rho} \Big| + |I_5| + |I_6| +|I_7|+|I_8|$$
where

\begin{eqnarray*}
I_5 & = & \frac{1}{2 \pi i} \int_{-U+iT}^{-U+iT_0} - \frac{\zeta'(s)}{\zeta(s)} \frac{x^s}{s} ds \hspace{0.7in} I_6  =   \frac{1}{2 \pi i}  \int_{-U+iT_0}^{-1+iT_0} - \frac{\zeta'(s)}{\zeta(s)} \frac{x^s}{s} ds\\ \\
I_7 & = &  \frac{1}{2 \pi i}  \int_{-1+iT_0}^{c+iT_0} - \frac{\zeta'(s)}{\zeta(s)} \frac{x^s}{s} ds \hspace{0.8in} I_8  =   \frac{1}{2 \pi i}  \int_{c+iT}^{c+iT_0} - \frac{\zeta'(s)}{\zeta(s)} \frac{x^s}{s} ds.
\end{eqnarray*}
From $(\ref{boundzerocount})$, we can estimate the sum by

$$\sum_{T-1 < \Im \rho < T+1} \Big| \frac{x^{\rho}}{\rho} \Big| < \sum_{T-1 < \Im \rho < T+1} \frac{x}{T-1} < \frac{2 x \log T}{T-1}.$$

We can bound $I_5$ in the same way as $I_3$ to obtain

$$|I_5| <\frac{18+2\log \sqrt{U^2+(T+1)^2}}{2 \pi x^U T}.$$

Bounding $I_6$ is done using Lemma \ref{bound1}:

$$|I_6| < \frac{9+\log \sqrt{U^2+(T+1)^2}}{2 \pi x (T-1)}$$

We also use Lemma \ref{choicet} to get

$$|I_7| < \frac{e}{2 \pi (T-1)} (\log^2 (T+1) + \log (T+1)).$$

To get an upper bound for $I_8$, we notice that $\Re s = 1+ 1/\log x$ and so following the line of working which led to $(\ref{s1s5})$ gives

$$\Big| \frac{\zeta'(s)}{\zeta(s)} \Big| < \log x.$$ 
This is by the working involved in $(\ref{s1s5})$. Following through we get the bound

$$|I_8| < \frac{ex \log x}{\pi(T-1)}.$$

Now, throwing all of our estimates for the terms in $(\ref{bigerror})$ together, implanting the information that $T \leq x$, $x > e^{60}$ and letting $U$ be equal to the even integer closest to $x$ we obtain Theorem \ref{explicitformula}. It now remains to apply this result to the problem of primes between cubes.

\section{Primes between cubes}

\subsection{An indicator function for intervals}

We define the Chebyshev $\theta$-function as

$$\theta(x) = \sum_{p \leq x} \log p$$
and consider that 

$$\theta_{x,h} = \theta(x+h)-\theta(x) = \sum_{x < p \leq x+h} \log p$$
is positive if and only if there is at least one prime in the interval $(x,x+h]$. Many questions involving the primes can be phrased in terms of $\theta_{x,h}$. For example, the twin prime conjecture is equivalent to $\theta_{p,2}$ taking on a positive value infinitely often where $p$ is a prime. 

We are interested in setting $h = 3 x^{2/3}$ to tackle the problem of primes between cubes. For if 

$$\theta_{x,3x^{2/3}} > 0$$
for all $x > x_0$, then there is a prime in the interval

$$(x,x+3x^{2/3}]$$
for all $x > x_0$. If we then set $x = y^3$ then there is a prime in the interval

$$(y^3, y^3 + 3 y^{2}] \subset (y^3, (y+1)^3)$$
for all $y > y_0 = x_0^{1/3}$. It is our intention to determine explicitly a value for $x_0$ and thus $y_0$. 

We introduce Chebyshev's $\psi$-function, given by

$$\psi(x) = \sum_{p^r \leq x} \log p.$$ 
Let $h$ be some positive real number. Substituting $x+h$ and then $x$ into Theorem \ref{explicitformula} and taking the difference, we find that: 

\begin{equation} \label{psiinterval}
\psi(x+h) - \psi(x) > h - \bigg| \sum_{|\gamma| < T} \frac{(x+h)^{\rho} - x^{\rho}}{\rho}\bigg| - \frac{4 (x+h) \log^2 (x+h)}{T}.
\end{equation}

Whilst the above will tell us information about prime powers, we are actually interested in primes. We thus require the following lemma; a combination of Proposition 3.1 and 3.2 of Dusart \cite{dusart}.

\begin{lem} 
Let $x \geq 121$. Then

$$0.9999  x^{1/2} < \psi(x) - \theta(x) < 1.00007 x^{1/2} + 1.78 x^{1/3}.$$ 
\end{lem}

It follows from (\ref{psi}) and the above lemma that

\begin{eqnarray} \label{big}
\theta_{x,h} & > & h - \bigg| \sum_{|\gamma| < T} \frac{(x+h)^{\rho} - x^{\rho}}{\rho}\bigg| - \frac{4 (x+h) \log^2 (x+h)}{T}  \nonumber \\ \nonumber \\
&-& 1.00007 (x+h)^{1/2}-1.78 (x+h)^{1/3} + 0.9999 x^{1/2}.
\end{eqnarray}
Given that we are interested in $h = 3 x^{2/3}$, it remains to choose $T=T(x)$ and find $x_0$ such that $\theta_{x,h}$ is positive for all $x>x_0$. 

\subsection{Estimating the sum over the zeroes}

In consideration of (\ref{big}), we let 

$$S = \bigg| \sum_{|\gamma| < T} \frac{(x+h)^{\rho} - x^{\rho}}{\rho}\bigg|.$$
We then have that

\begin{eqnarray*}
S & = & \bigg| \sum_{|\gamma| < T} \int_{x}^{x+h} t^{\rho-1} \bigg| \leq  \sum_{|\gamma| < T} \int_{x}^{x+h} t^{\beta-1}  \leq h \sum_{| \gamma| < T} x^{\beta-1}.
\end{eqnarray*}
From the identity

\begin{eqnarray*}
\sum_{|\gamma| < T} ( x^{\beta-1} - x^{-1}) & = & \sum_{|\gamma| < T} \int_0^{\beta} x^{\sigma - 1} \log x d \sigma \\
& = & \int_0^1 \sum_{ \beta> \sigma, |\gamma| < T } x^{\sigma-1} \log x d \sigma,
\end{eqnarray*}
it follows that

\begin{equation} \label{summypoos}
\sum_{|\gamma| < T} x^{\beta  -1} = 2 x^{-1} N(T) + 2 \int_0^{1} N(\sigma,T) x^{\sigma-1} \log x d \sigma,
\end{equation}

where $N(\sigma, T)$ denotes the number of zeroes $\rho$ of the Riemann zeta-function with $0<\Im \rho < T$ and $\Re(\rho) > \sigma$.

We can estimate the above sum, and thus $S$, with the assistance of some explicit bounds. Firstly note, that by Corollary 1 of Trudgian \cite{trudgianargument} we have that

$$N(T) < \frac{T \log T}{2 \pi} $$
for all $T>15$, say. Explicit estimates for $N(\sigma,T)$ are rare, though have come to light recently through the likes of Kadiri \cite{kadiri} and Ramar\'{e} \cite{ramare}, who have produced zero-density estimates of rather different shape to each other. Ramar\'{e}'s estimate, which is an explicit and asymptotically better version of Ingham's \cite{ingham} original density estimate, is required for the problem of primes between cubes. We give the result here, which is a corollary of Theorem 1.1 of \cite{ramare}.

\begin{lem} \label{ramaredensity}
Let $T \geq 2000$ and $\sigma \geq 0.52$. Then

$$N(\sigma,T) \leq 9.7 (3T)^{8(1-\sigma)/3} \log^{5-2\sigma} T + 103 \log^2 T.$$
\end{lem}

The following zero-free region, given by Ford, will also be required.

\begin{lem} \label{fordregion}
Let $T \geq 3$. Then there are no zeroes of $\zeta(s)$ in the region given by $\sigma \geq 1 - \nu(T)$ where

$$\nu(T) = \frac{1}{57.54 \log^{2/3} T (\log \log T)^{1/3}}.$$

\end{lem}

It is useful to carry out the bulk of the calculations with $A$ in place of the constant 9.7 in Lemma \ref{ramaredensity} and $c$ in place of the $57.54$ in Lemma \ref{fordregion}. Doing so allows us later on to see the importance of improvements of these constants, and thus gives direction to future efforts on this problem.

We split the integral in (\ref{summypoos}) into two parts, one over the interval $0 \leq \sigma \leq 5/8$, where $N(\sigma,T)$ may as well be bounded by $N(T)$, and another over $5/8 \leq \sigma \leq 1-\nu(T)$. By inserting the relevant estimates, we get

\begin{eqnarray} \label{crackers}
\sum_{|\gamma| < T} x^{\beta-1} & < & 2 x^{-1} N(T) + 2 x^{-1} N(T) \log x \int_0^{5/8} x^{\sigma} dx \nonumber  \\
& + & 2 A x^{-1} (3T)^{8/3} \log x \log^5 T \int_{5/8}^{1-\nu(T)} \bigg( \frac{x}{(3T)^{8/3} \log^2 T} \bigg)^{\sigma} d \sigma  \nonumber \\
& + & 103 x^{-1} \log x \log^2 T \int_{5/8}^{1-\nu(T)} x^{\sigma} d \sigma.
\end{eqnarray}

The working out is routine, yet tedious. We give the qualitative details to the extent that the reader can follow the process. We introduce the parameter $k \in (\frac{2}{3},1)$, which will play a part in the relationship between $T$ and $x$. The reasons for the range of values of $k$ will become clear soon. Now let $T=T(x)$ be the solution to the equation

$$\frac{x}{(3T)^{8/3} \log^2 T} = \exp( \log^k x).$$

Upon performing the integration in (\ref{crackers}), we directly substitute in the above relationship, along with the bound for $N(T)$ and the fact that $\log T < (3/8) \log x$, to get

\begin{eqnarray} \label{dong}
\sum_{|\gamma| < T} x^{\beta-1} & < & \frac{ e^{-\frac{3}{8} \log^k x} \log^{1/4} x}{3^{3/4} 8^{1/4} \pi}+ \frac{27A}{256} \log^{4-k} x ( e^{-\nu(T) \log^k x}- e^{-(3/8) \log^k x}) \nonumber \\ 
& + & \frac{927 A}{32} \log^2 x (e^{- \nu(T) \log x} - x^{-3/8}).
\end{eqnarray}
There is some cancellation in the above. First, we need to estimate one of the exponential terms involving $\nu(T)$. We have that

\begin{eqnarray*}
e^{-v(T) \log x} & = & \exp\Big( -  \frac{\log x}{c \log^{2/3} T (\log \log T)^{1/3}} \Big) \\
& < &  \exp\Big( -  \frac{4}{3^{2/3} c} \Big( \frac{\log x}{  \log \log x} \Big)^{1/3} \Big).
\end{eqnarray*}
Now, upon expansion of (\ref{dong}) and using the above we can notice that

$$- \frac{27 A}{256} (\log x)^{4-k} e^{- (3/8) \log^k x}+ \frac{927 A}{32} \log^2 x (e^{- \nu(T) \log x} - x^{-3/8}) < 0.$$
This is clear if one looks at the dominant terms. It follows that

\begin{eqnarray} 
\sum_{|\gamma| < T} x^{\beta-1} & < & \frac{ e^{-\frac{3}{8} \log^k x} \log^{1/4} x}{3^{3/4} 8^{1/4} \pi}+ \frac{27A}{256} (\log x)^{4-k} e^{-\nu(T) \log^k x}.
\end{eqnarray}
The remaining exponential term involving $\nu(T)$ is dealt with as before to get

\begin{eqnarray*} 
S \leq h \sum_{|\gamma| < T} x^{\beta-1} & < & \frac{ h e^{-\frac{3}{8} \log^k x} \log^{1/4} x}{3^{3/4} 8^{1/4} \pi}+ \frac{27Ah}{256} (\log x)^{4-k} \exp\Big(- \frac{4}{3^{2/3}c}  \frac{\log^{k-2/3} x}{(\log \log x)^{1/3}} \Big).
\end{eqnarray*}

\subsection{Estimates for inequalities}

It is now clear that we may write $(\ref{big})$ as

$$\theta_{x,h} > h - f(x,h,k,A,c) - g(x,h,k)-E(x,h,k)$$
where

\begin{eqnarray*}
f(x,h,k,A,c) & = & \frac{27 A h}{256} (\log x)^{4-k} \exp\Big( - \frac{4}{3^{2/3} c} \frac{\log^{k-2/3} x}{(\log \log x)^{1/3}}\Big),\\ \\
g(x,h,k) & = & 12 \Big(\frac{3}{8} \Big)^{3/4} \frac{(x+h) \log^{11/4} (x+h)}{x^{3/8}} \exp(\frac{3}{8} \log^k x),\\ \\
 E(x,h,k) & = & - \frac{ h (\log x)^{1/4} \exp(-\frac{3}{8} \log^k x)}{6^{3/4} \pi}-1.00007(x+h)^{1/2} \\\
 &-&1.78 (x+h)^{1/3}+0.9999 x^{1/2}.
\end{eqnarray*}

First, we look to bound the error. Noting that $x > e^{60}$, we set $h = 3x^{2/3}$ and use the fact that $k=2/3$ will give us the worst possible error to get

$$\frac{E(x,3 x^{2/3},2/3)}{3x^{2/3}} < 10^{-3}.$$
Thus, one can show that positivity holds if the following two inequalities are simultaneously satisfied:

\begin{enumerate}
\item $f(x,h,k,A,c) < \frac{1}{2} (1-10^{-3}) h$,
\item $g(x,h,k) < \frac{1}{2} (1-10^{-3}) h.$
\end{enumerate}
This splitting simplifies our working greatly whilst perturbing the solution negligibly. To be convinced of this, one could consider the right hand side of each of the above inequalities as being equal to $h$, in some better-than-possible scenario. It turns out that the improvements would hardly be noticeable. We will, however, mention at the end of this paper some direction for future attempts at improving the work on primes between cubes.

Now, in the first inequality, we take the logarithm of both sides and set $x = e^y$ to get

\begin{equation}\label{1}
\log\Big(\frac{27 A}{256}\Big) + (4-k) \log y - \frac{4}{3^{2/3} c} \frac{y^{k-2/3}}{\log^{1/3} y} < \log (\frac{1}{2} (1-10^{-3}))
\end{equation}
This is easy to solve using \textsc{Mathematica}, given knowledge of $A$, $k$ and $c$. There are some notes to make here first. We can see that $A$, the constant in front of Ramar\'{e}'s zero-density estimate has little contribution, for being in the argument of the logarithm. However, $c$ plays a much larger part from where it is positioned. We can also see the reason for $k > 2/3$, in that it guarantees a solution.

We deal with the second inequality in the same way, but first we notice that

$$\frac{g(x,h,k)}{h} < \frac{2 \log^{11/4} x}{x^{1/24}} \exp( \frac{3}{8} \log^k x).$$
This is obtained using the main result of Ramar\'{e} and Saouter \cite{ramaresaouter} to bound

$$x+h < \frac{x}{1-\Delta^{-1}}$$
where $\Delta = 28 314 000$ as given in their paper. Thus, using the same approach as before we get

\begin{equation} \label{2}
\frac{11}{4} \log y + \frac{3}{8} y^k-\frac{1}{24} y < \log( \frac{1}{4} (1-10^{-3})).
\end{equation}

We notice here our reason for having $k<1$. One can also see the reason for leaving $k$ free to vary in $(2/3,1)$. There should be an optimal value of $k$, where the solution range of the above two inequalities are equal and their intersection is minimised. 

No rigorous analysis needs to be conducted; we set $A=9.7$, $c=57.54$ and use the \texttt{Manipulate} function of \textsc{Mathematica} to ``hunt'' for a good value of $k$. It turns out that upon choosing $k=0.9359$, we have that both inequalities are satisfied for $y>8 \times 10^{14}$, or $x^{1/3} > \exp(\exp(33.217))$, which proves our main result.

\subsection{Notes for future improvements}

Using the explicit methods of this paper, better estimates for zero-densities, zero-free regions and the error term of Landau's explicit formula could effectively be implemented to furnish a new estimate. The following preemptive discussion might be useful for one looking to do such a thing.

Let's consider first improving the zero-density estimate given by Ramar\'{e}. Say, for the sake of discussion, one could obtain a value of $A=10^{-4}$. Then we would obtain our result instead with $n \geq \exp(\exp(32.7))$, an improvement which would probably not be worth the efforts required to obtain such a value of $A$. 

Ramar\'{e} has communicated that one could use the Brun-Titchmarsh theorem to remove a power from the logarithm in the error term of (\ref{explicitformula}). This does not seem to improve the overall result; a shortcoming, perhaps, of the numerical methods used by the author.  

There are other parameters where one might wish to direct future efforts. In Ramar\'{e}'s zero density estimate, one might consider the power $5-2\sigma$ of the logarithm to be $L-2\sigma$. The main difference in our working would be $(L-1-k)$ in place of $(4-k)$ in the reduced form of our second inequality. The following table summarises the improvements which would follow; a prime between $n^3$ and $(n+1)^3$ for all $n \geq n_0$.

\begin{center}
  \begin{tabular}{ | c | c | }
    \hline
    $L$  & $\log \log n_0$ \\ \hline \hline
    5 & 33.217 \\ \hline
    4 & 31.8 \\ \hline
    3 & 29.8 \\ \hline
    2 & 22.19 \\ \hline
    
  \end{tabular}
\end{center}

Turning now to the error term of Theorem \ref{explicitformula} one could also consider a smaller constant in place of $2$. This constant, however, would appear in the logarithm of the right hand side of (\ref{2}), and thus make little difference. 

Wolke has derived the explicit formula with an error term which is

$$O\Big(\frac{x \log x}{T \log (x/T)}\Big) = O\Big( \frac{x}{T} \Big)$$
for the choice of $T(x)$ used in this paper. One may propose all sorts of ``good'' explicit constants for the above error term and try them via the methods of this paper, but there will be no major improvements.

Changes in the constant $c$ are more effective, though seemingly much more difficult to obtain. A value of $c=40$ would yield only $n \geq \exp(\exp(31.88))$, and $c=20$ would give $n \geq \exp(\exp(29.6))$. The removal of the $(\log \log T)^{1/3}$ would give a similar result.

Thus one expects a major result, or perhaps many minor ones, to make significant progress on this problem.

\subsection{Higher powers}

In lieu of a complete result on the problem of primes between cubes, we consider instead primes between $m$th powers, where $m$ is some positive integer. Appropriately, we choose $h = m x^{1-1/m}$, and we are able to prove the following result.

\begin{thm} \label{mpowers}
Let $m \geq 4.971 \times 10^9$. Then there is a prime between $n^m$ and $(n+1)^m$ for all $n \geq 1$.
\end{thm}

The result seems absurd on a first glance as the value of $m$ is quite large. We shall leave it to others to attempt to bring the value down. 

We now prove the above theorem as follows; for our choice of $h$, it follows that inequality (\ref{2}) becomes

\begin{equation} \label{new2}
\frac{11}{4} \log y - \Big( \frac{3}{8} - \frac{1}{m} \Big) y+\frac{3}{8} y^k < \log\Big( \frac{m}{12} (1-10^{-3})\Big)
\end{equation}
whereas inequality (\ref{1}) remains the same. As before, we can, for some given $m$, choose $k$ and find $n_0$ such that there is a prime between $n^m$ and $(n+1)^m$ for all $n \geq n_0$ by solving both inequalities. Some results are given in the following table.

\begin{center}
  \begin{tabular}{ | c | c | c | }
    \hline
    $m$  & $k$ & $\log \log n_0$ \\ \hline \hline
    4 & 0.9635 & 29.240 \\ \hline
    5 & 0.9741 & 27.820 \\ \hline
    6 & 0.9796 & 27.230 \\ \hline
    7 & 0.983 & 26.427 \\ \hline
    1000 & 0.9998 & 19.807 \\ \hline
  \end{tabular}
\end{center}

One can see that this method has its limitations, even in the case of higher powers. Nonetheless, we have that there is a prime in $(n^{1000}, (n+1)^{1000})$ for all $n  \geq \exp(\exp(19.807))$. It follows that, for $m \geq 1000$, there is a prime between $n^m$ and $(n+1)^m$ for all 

\begin{equation} \label{lower}
n \geq \exp\bigg( \frac{1000 \exp(19.807)}{m}\bigg).
\end{equation}
We could choose $m = 1000 \exp(19.807) \approx 4 \times 10^{11}$ to get primes between $n^m$ and $(n+1)^m$ for all $n \geq e$. Betrand's postulate improves this to all $n \geq 1$.

However, we can use Corollary 2 of Trudgian \cite{trudgianpomerance} to improve on this value of $m$. This states that for all $x \geq 2898239$ there exists a prime in the interval 

$$\bigg[ x,x\bigg(1+\frac{1}{111\log^2 x}\bigg) \bigg].$$
If we set $x = n^m$, we might ask when the above interval falls into $[n^m,n^m+m n^{m-1}]$. One can rearrange the inequality 

$$n^m \bigg(1+\frac{1}{111\log^2 (n^m)}\bigg) < n^m+m n^{m-1}$$
to get

\begin{equation} \label{upper}
\frac{n}{\log^2 n} < 111 m^3.
\end{equation}
We wish to choose the lowest value of $m$ for which the solution sets of (\ref{lower}) and (\ref{upper}) first coincide. It is not to hard to see that this equates to solving simultaneously the equations

\begin{equation*}
n = \exp\bigg( \frac{1000 \exp(19.807)}{m}\bigg)
\end{equation*}
and

\begin{equation*} 
\frac{n}{\log^2 n} = 111 m^3.
\end{equation*}
We do this by substituting the first equation directly into the second to get

$$\exp\bigg( \frac{1000 \exp(19.807)}{m}\bigg) = 111 (1000 \exp(19.807))^2 m$$
which can easily be solved with \textsc{Mathematica} to prove Theorem \ref{mpowers}. 

\clearpage

\bibliographystyle{plain}

\bibliography{biblio}

\end{document}